\documentclass[12pt,draftcls,onecolumn]{IEEEtran}

\IEEEoverridecommandlockouts
\usepackage{times,color,pslatex}
\usepackage{amssymb,amsmath,amsfonts,latexsym,amscd,amsthm}
\usepackage{graphicx,graphics,epsfig,subfigure}

\newcommand{\ip}[2]{\langle #1,#2 \rangle}
\newcommand{\eq}[1]{\mbox{{\rm(\hspace{-.35em}~\ref{#1})}}}
\def\ds{\displaystyle}

\newtheorem{lem}{Lemma}
\newtheorem{rem}{Remark}
\newtheorem{assum}{Assumption}

\title{Adaptation and optimization of synchronization gains in networked
distributed parameter systems}

\author{Michael A. Demetriou
\thanks{The author is with
Aerospace Engineering Program,
Mechanical Engineering Dept,
Worcester Polytechnic Institute,
Worcester, MA 01609, USA, {\tt\small mdemetri@wpi.edu}.
The author gratefully acknowledges financial support from
the AFOSR, grant FA9550-12-1-0114.}%
}

\begin{document}

\maketitle

\begin{abstract}
This work is concerned with the design and effects of the synchronization gains on the
synchronization problem for a class of networked distributed parameter systems.
The networked systems, assumed to be described by the same evolution
equation in a Hilbert space, differ in their initial conditions.
The proposed synchronization controllers aim at achieving both
the control objective and the synchronization objective. 
To enhance the synchronization, as measured by the norm of the
pairwise state difference of the networked systems, an adaptation of the 
gains is proposed. An alternative design arrives at constant gains
that are optimized with respect to an appropriate measure of synchronization.
A subsequent formulation casts the control and synchronization design problem into
an optimal control problem for the aggregate systems.
An extensive numerical study examines the various aspects of the optimization
and adaptation of the gains on the control and synchronization
of networked 1D parabolic differential equations. 
\end{abstract}

\begin{keywords}
Distributed parameter systems;
distributed interacting controllers;
networked systems;
adaptive synchronization
\end{keywords}

\abovedisplayskip=1ex 
\belowdisplayskip=1ex 
\abovedisplayshortskip=1ex 
\belowdisplayshortskip=1ex 

\section{Introduction} \label{sec1}

The problem of synchronization of dynamical systems witnessed a surge
of interest in the last few years, primarily for finite dimensional systems
\cite{SuBook2013,WuIEEECS95,LewisTAC11,SaberiIJRNC11,StoorvogelACC11}.
Adaptive and robust control techniques were considered primarily for system
with linear dynamics. A special case of nonlinear systems, the Lagrangian
systems which describe mobile robots and spacecraft,
also considered aspects of synchronization control
\cite{LewisSMC11,Slotine2009,ShanIET08,Min2012,KrogstadCDC06,RenIJC09,Chung2009}.

For distributed parameter systems (DPS), fewer results can be found \cite{DemetriouSCL2013,DemetriouAutomatica2010,WuIEEETCS12,Wang2012,Ambrosio12,Rao13}.
In \cite{WuIEEETCS12}, a system of coupled diffusion-advection PDEs 
was considered and conditions were provided for their synchronization. 
In a similar
fashion \cite{Ambrosio12} considered coupled reaction–diffusion systems of
the FitzHugh-Nagumo type and classify their stability and synchronization.
In the same vein, \cite{Rao13} examined coupled hyperbolic PDEs and through
boundary control proposed a synchronization scheme.
Somewhat different spin but with essentially a similar
framework of coupled PDEs was considered in \cite{Wang2012},
where an array of linearly coupled neural
networks with reaction-diffusion terms and delays were considered.
However, designing a synchronizing control law for uncoupled PDE systems
has not appeared till recently \cite{DemetriouSCL2013}, where a
special class of PDEs, namely those with a Riesz-spectral state operator,
were considered. An unresolved problem is that of a network of uncoupled PDE systems
interacting via an appropriate communication topology. Further, the choice and optimization
of the synchronization gains has not been addressed. Such an unsolved problem is being
considered here.

The objective of this note is to extend the use of the edge-dependent scheme 
to a class of DPS. The proposed controllers,
parameterized by the edge-dependent gains which are associated with the 
elements of the Laplacian matrix of the graph topology, 
are examined in the context of optimization and adaptation. 
One component of the proposed linear controllers is responsible for the control
objective, assumed here to be regulation. The other component, which is  used
for enforcing synchronization, includes the weighted pairwise state differences. 
When penalizing the disagreement of the networked states, 
one chooses the weights in proportion to their disagreement. This can be done
when viewing all the networked systems collectively by optimally
choosing all the weights, or by adjusting these gains adaptively.
The contribution of this work is twofold:
\begin{itemize}
\item It proposes the optimization of the synchronization gains,
by considering the aggregate closed-loop systems and 
minimizes an appropriate measure of synchronization. Additionally,
it casts the control and synchronization design into an optimal
control problem for the aggregate systems with an LQR cost functional.

\item It provides a Lyapunov-based adaptation of the synchronization gains
as a means of improving the synchronization amongst a class of networked
distributed parameter systems described by infinite dimensional systems.
\end{itemize} 
The outline of the manuscript is as follows. The class of systems under consideration
is presented in Section~\ref{sec2}. The synchronization and control design objectives
are also presented in Section~\ref{sec2}. The main results on the choice
of adaptive and constant edge-dependent synchronization gains, including
well-posedness and convergence of the resulting closed-loop systems are given in Section~\ref{sec3}.
Numerical studies for both constant and adaptive gains are presented in Section~\ref{sec4}
with conclusions following in Section~\ref{sec5}.

\section{Mathematical framework and problem formulation}\label{sec2}

We consider the following class of infinite dimensional systems
with identical dynamics but with different initial conditions
on the state space $\mathcal{H}$
\begin{equation}\label{eq1}
\dot{x}_{i}(t) = A x_{i}(t) +  B_{2} u_{i}(t),
\hspace{0.5em} x_{i}(0)=x_{i0}\in \mathcal{D}(A),
\end{equation}
for $i=1\ldots,N$.
The state space $\{ \mathcal{H}, \ip{\cdot}{\cdot}_{\mathcal{H}},|\cdot|_{\mathcal{H}}\}$ is a Hilbert
space, \cite{Lions71}. To allow for a wider class of state and possibly input and output operators,
we formulate the problem in a space setting associated with a Gelf'and triple \cite{Showalter}.
Let $\{\mathcal{V}, \|\cdot\|_{\mathcal{V}}\}$ be a reflexive Banach space that is
densely and continuously embedded in $\mathcal{H}$ with
$\mathcal{V} \hookrightarrow \mathcal{H} \hookrightarrow \mathcal{V}^{*}$
with the embeddings dense and continuous where $\mathcal{V}^{*}$ denotes
the continuous dual of $\mathcal{V}$, \cite{Showalter}. The input space
$\mathcal{U}$ is a finite dimensional Euclidean space of controls.
In view of the above, we have that the state operator $A \in  \mathcal{L}(\mathcal{V},\mathcal{V}^{*})$
and the input operator $B_{2} \in  \mathcal{L}(\mathcal{U},\mathcal{H})$.

The \emph{synchronization objective} is to choose the control signals $u_{i}$,
$i=1,\ldots,N$, so that all pairwise differences asymptotically converge (in norm) to zero
\begin{equation} \label{eq2}
\lim_{t\rightarrow \infty} |x_{i}(t)-x_{j}(t)|_{\mathcal{H}} =0, \hspace{1em} \forall i,j=1,\ldots,N.
\end{equation}
An alternative and weaker convergence may consider \emph{weak synchronization} via
$$
\lim_{t\rightarrow \infty} \ip{x_{i}(t)-x_{j}(t)}{\varphi}_{\mathcal{V},\mathcal{V}^{*}}=0,
\hspace{0.25em} \forall i,j=1,\ldots,N, \hspace{0.25em} \varphi \in \mathcal{V}.
$$
An appropriate measure of synchronization is the \emph{deviation from the mean}
\begin{equation} \label{eq3}
z_{i}(t) = x_{i}(t) - \frac{1}{N}\sum_{j=1}^{\infty} x_{j}(t),
\hspace{1em} i=1,\ldots,N,
\end{equation}
which can also be viewed as the output-to-be-controlled \cite{LiuACC09}, and which measures
the disagreement of state $x_{i}(t)$ to the average state of all agents. It is easlily
observed that asymptotic norm convergence of each $z_{i}(t)$, $i=1,\ldots,N$ to zero implies
asymptotic norm convergence of all pairwise differences $x_{i}(t)-x_{j}(t)$ to zero
and vice-versa.

When examining the well-posedness of the $N$ systems, one must consider
them collectively. This motivates the definition of the state space
$\mathbb{H}= (\mathcal{H})^{N}$.
The spaces $\mathbb{V}$ and $\mathbb{V}^{*}$
are similarly defined via $\mathbb{V}= (\mathcal{V})^{N}$
and $\mathbb{V}^{*}= (\mathcal{V}^{*} )^{N} $ with
$\mathbb{V} \hookrightarrow \mathbb{H} \hookrightarrow \mathbb{V}^{*}$. Similarly,
define the space $\mathbb{U}= (\mathcal{U})^{N}$.

An undirected graph $G =(V,E)$ is assumed to describe the communication topology for the
$N$ networked PDE systems. The nodes $V =\{1,2,\ldots,N\}$ represent the agents (PDE systems)
and the edges $E \subset V\times V$ represent the communication links between the
networked systems \eqref{eq1}. The set of systems (neighbors) that the $i$th system
is communicating with is denoted by $N_{i} = \{ j: (i,j) \in E\}$, \cite{Godsil2001}.

The parameter space $\Theta \in \mathbb{R}^{N\times N}$ is defined
as the space of $N\times N$ (Laplacian) matrices $L$ with the property that
$L_{ii}=-\sum_{j \in N_{i}}^{N}L_{ij}>0$, $i=1,\ldots,N$, i.e. we have
$$
\Theta = \big\{ L \in \mathbb{R}^{N\times N} \, : \,  L_{ii}=-\sum_{j \in N_{i}}^{N}L_{ij} >0 \big\}.
$$
The space $\{\mathbb{H},\ip{\cdot}{\cdot}_{\mathbb{H}} \}$ is a Hilbert space with inner product
$$
\ip{\Phi}{\Psi}_{\mathbb{H}} = \ip{\phi_{1}}{\psi_{1}}_{\mathcal{H}}
+ \ip{\phi_{2}}{\psi_{2}}_{\mathcal{H}} + \ldots  + \ip{\phi_{N}}{\psi_{N}}_{\mathcal{H}},
$$
with $\Phi=\{\phi_{1},\phi_{2},\ldots , \phi_{N}\}, \Psi=\{\psi_{1},\psi_{2},\ldots , \psi_{N}\} \in \mathbb{H}$.
In view of the above, the deviation from the mean \eqref{eq3} can be written in terms of
the aggregate state vector and the aggregate deviation from the mean as
\begin{equation} \label{eq4}
Z(t) = C_{1}X(t),
\end{equation}
where
$X(t) =  [\begin{array}{ccc}
x_{1}(t) &
\hdots &
x_{N}(t)  \end{array} ]^{T}$,
$Z(t) = [\begin{array}{ccc}
z_{1}(t) &
\hdots &
z_{N}(t)  
\end{array}]^{T}$, 
$C_{1}= \mathbf{I}_{N}-\frac{1}{N} \mathbf{1}_{N}\cdot \mathbf{1}_{N}^{T}$,
where $\mathbf{I}_{N}$ denotes the $N$-dimensional identity matrix understood in the sense
of each entry being the identity operator on $\mathcal{H}$. Similarly, $\mathbf{1}_{N}$
denotes the $N$-dimensional column vector of $1$'s, similarly understood in the sense
of $\mathbf{1}_{N}\cdot \mathbf{1}_{N}^{T}$ being the $N\times N$ matrix whose entries
are the identity operator on $\mathcal{H}$. The matrix operator $C_{1}$ corresponds to the graph
Laplacian matrix operator with all-to-all connectivity with $NC_{1}=L$.
In view of this, the synchronization objective in \eqref{eq2} can equivalently be stated as
$\lim_{t\rightarrow \infty} |Z(t)|_{\mathbb{H}} = 0$,
and together with the control objective, assumed here to be state regulation, is
combined to give rise to the design objective of the networked systems \eqref{eq1}.

\noindent \textbf{Design objectives:} Design control signals for the networked systems \eqref{eq1} such that
\begin{equation} \label{eq5}
\left\{
\begin{array}{lcl}
\ds \lim_{t\rightarrow \infty} |X(t)|_{\mathbb{H}} = 0, 
&&  \mbox{(\textbf{state regulation})}
\\ \noalign{\medskip}
\ds \lim_{t\rightarrow \infty} |Z(t)|_{\mathbb{H}} = 0,
&&
\mbox{(\textbf{synchronization})}  
\end{array} \right.
\end{equation}
\begin{rem}
Please notice that regulation of $X(t)$ to zero (in $\mathbb{H}$ norm)
immediately implies synchronization, but the converse
cannot be guaranteed. Careful examination of $Z(t)=C_{1}X(t)$ sheds light to this case, since
the matrix operator $C_{1}$, which corresponds to the graph Laplacian with all-to-all connectivity,
has a zero eigenvalue.
\end{rem}

\section{Main results: edge-dependent synchronization gains} \label{sec3}

The $N$ systems in \eqref{eq1} are considered
with each state $x_{i}$ available. A leaderless configuration
is assumed and thus each agent  will only access the states of its neighboring agents
as dictated by the communication topology.
A standing assumption for the systems in \eqref{eq1} is now presented.
\begin{assum} \label{assum1}
Consider the networked systems in \eqref{eq1}. Assume the following
\begin{enumerate}
\item The state $x_{i}(t)$ of each system is available to the $i$th system and also
to all the other networked systems that is linked to as dictated by the communication
topology, assumed here to be described by an undirected connected graph.

\item The operator $A$ generates a $C_{0}$ semigroup on $\mathcal{H}$
and for any $u_{i}\in L_{2}(0,\infty)$, the systems \eqref{eq1} are well-posed
for any $x_{i}(0)\in D(A)$.

\item The pair $(A,B_{2})$ is approximately
controllable\footnote{Normally, one would require that the pair $(A,B_{2})$ be exponentially
stabilizable. When the operator $A$ generates an exponentially stable $C_{0}$ semigroup, then
one only requires approximate controllability \cite{Curtain95}.},
i.e. there exists a feedback gain operator $K\in \mathcal{L}(\mathcal{H},\mathcal{U})$
such that the operator
$A_{c} \triangleq A - B_{2}K$,
generates an exponentially stable $C_{0}$ semigroup, with the property that
\begin{equation} \label{eq6}
A_{c}+A_{c}^{*} \leq -\kappa I, \hspace{1em}
\kappa >0.
\end{equation}
\end{enumerate}
\end{assum}
The operator equation \eqref{eq6} is a simplified version
of the operator Lyapunov function \cite{Tucsnak09}.

For simplicity, denote the differences of $x_{i}$ and $x_{j}$ by
$x_{ij}(t) \triangleq x_{i}(t)-x_{j}(t)$, $j\in N_{i}$, $i=1,\ldots,N$.
The controllers with \emph{constant edge-dependent synchronization gains}
$\alpha_{ij}$ are given by
\begin{equation} \label{eq7}
u_{i}(t) = -Kx_{i}(t) - F \sum_{j\in N_{i}}\alpha_{ij} x_{ij}(t), \hspace{0.5em} i=1,\ldots,N.
\end{equation}
whereas with \emph{adaptive edge-dependent synchronization gains} $\alpha_{ij}(t)$ are given by
\begin{equation} \label{eq8}
u_{i}(t) = -Kx_{i}(t) - F \sum_{j\in N_{i}}\alpha_{ij}(t) x_{ij}(t), \hspace{0.5em} i=1,\ldots,N.
\end{equation}

The control signals consist of the local controller used to achieve the control objective (regulation)
and a networked component required for enforcing synchronization. The feedback operator
$K\in \mathcal{L}(\mathcal{H},\mathcal{U})$ is chosen so that $(A-B_{2}K)$ generates an exponentially
stable $C_{0}$ semigroup on $\mathcal{H}$ and the synchronization gain $F\in \mathcal{L}(\mathcal{H},\mathcal{U})$ is chosen so that certain synchronization conditions
are satisfied.

Both \eqref{eq7}, \eqref{eq8} will be considered below and different methods
for choosing the edge-dependent gains $\alpha_{ij}$ and $\alpha_{ij}(t)$ will
be described.
\subsection{Adaptive edge-dependent synchronization gains}

A way to enhance the synchronization of the networked systems, is to employ
adaptive strategies to tune the strengths of the network nodes interconnections
as was similarly addressed for finite dimensional systems
\cite{deLellisTAC12,DeLellisAutomatica09,DeLellisChaos08,DeLellisIEEE2010}.

In the case of adaptive synchronization gains, the closed-loop systems are given by
\begin{equation} \label{eq9}
\dot{x}_{i}(t)=(A-B_{2}K)x_{i}(t) - B_{2}F \sum_{j\in N_{i}}\alpha_{ij}(t) x_{ij}(t),
\hspace{1em} x_{i}(0)\in D(A),
\end{equation}
for $i=1,\ldots,N$.
To derive the adaptive laws for the edge-dependent gains, one considers the following
Lyapunov-like functionals
$$
V_{i}(x_{i},\alpha_{ij})=    |x_{i}(t)|^{2}_{\mathcal{H}} +   \sum_{j\in N_{i}} \alpha_{ij}^{2}(t),
\hspace{1em} i=1,\ldots,N.
$$
Using \eqref{eq6}, \eqref{eq9}, its time derivative is given by
$$
\dot{V}_{i}  =   \ds -\kappa |x_{i}|^{2}_{\mathcal{H}}  
+ 2\sum_{j\in N_{i}}\alpha_{ij}
\left(\dot{\alpha}_{ij} - \ip{x_{i}}{BFx_{ij}}_{\mathcal{H}} \right).
$$
While the choice $\dot{\alpha}_{ij} = \ip{x_{i}}{BFx_{ij}}_{\mathcal{H}}$
results in $\dot{V}_{i} \leq - \kappa |x_{i} |^{2}_{\mathcal{H}}$, one may consider
\begin{equation} \label{eq10}
\dot{\alpha}_{ij} = \ip{x_{i}}{BFx_{ij}}_{\mathcal{H}} - \sigma \alpha_{ij},
\hspace{0.25em} j\in  N_{i}, \hspace{0.25em} i=1,\ldots,N,
\end{equation}
where $\sigma > 0$ are the adaptive gains \cite{IS95}. This results in
$\dot{V}_{i} = -\kappa |x_{i}|^{2}_{\mathcal{H}}  
- 2\sigma \sum_{j\in N_{i}}\alpha_{ij}^{2}$,
$i=1,\ldots,N$. Summing from $i=1$ to $N$
$$
\sum_{i=1}^{N}\dot{V}_{i} = - \sum_{i=1}^{N}\big(  \kappa | x_{i}|_{\mathcal{H}}^{2}
+ 2 \sigma \sum_{j\in N_{i}}\alpha_{ij}^{2} \big)
\leq  -\min\{\kappa, 2 \sigma\}\sum_{i=1}^{N} V_{i} .
$$
For each $i=1,\ldots,N$, one can then show that $|x_{i}|_{\mathcal{H}} \rightarrow 0$
as $t\rightarrow \infty$ and for each $j\in N_{i}$, one also has $\alpha_{ij}\rightarrow 0$
as $t\rightarrow \infty$.

To examine the well-posedness and regularity of the closed loop systems,
the state equations \eqref{eq9} are written in aggregate form  
$$
\frac{\mathrm{d}}{\mathrm{d}t} X(t)   =
(\mathbf{I}_{N} \otimes A_{c}) X(t)
-\left[\begin{array}{c}
\ds B_{2}F \sum_{j\in N_{1}} \alpha_{1j}(t) x_{1j}(t)  \\ \noalign{\medskip}
\vdots \\ \noalign{\medskip}
\ds B_{2}F \sum_{j\in N_{N}} \alpha_{Nj}(t) x_{Nj}(t)
\end{array}\right], 
$$
with $X(0) \in D(\mathbf{I}_{N} \otimes A)$. 
To avoid over-parametrization, we express the adaptive edge-dependent gains $\alpha_{ij}(t)$
in terms of the elements $L_{ij}(t)$ of the time-varying graph Laplacian matrix and thus
$$
\frac{\mathrm{d}}{\mathrm{d}t} X(t)  =
(\mathbf{I}_{N} \otimes A_{c}) X(t)
-\left[\begin{array}{c}
\ds B_{2}F \sum_{j=1}^{N} L_{1j}(t) x_{j}(t)  \\ \noalign{\medskip}
\vdots \\ \noalign{\medskip}
\ds B_{2}F \sum_{j=1}^{N} L_{Nj}(t)x_{j}(t)
\end{array}\right], 
$$
with $X(0) \in D(\mathbf{I}_{N} \otimes A)$.
With this representation  one can write the above compactly as
$$
\frac{\mathrm{d}}{\mathrm{d}t} X(t)=\mathcal{A}_{c}X(t)
-  \mathcal{B}_{2} L(t)     \mathcal{F}  X(t) ,
\hspace{1em}X(0)\in D(\mathcal{A}),
$$
where
$\mathcal{A}=\mathbf{I}_{N} \otimes A$,
$\mathcal{A}_{c} \triangleq \mathbf{I}_{N} \otimes A_{c}$,
$\mathcal{B}_{2} \triangleq  \mathbf{I}_{N}\otimes B_{2}$,
$\mathcal{F} \triangleq \mathbf{I}_{N}\otimes F$.
Following the approach for $\alpha_{ij}(t)$ in \eqref{eq10}, 
the adaptation of the interconnection strengths (elements of $L(t)$) is given in weak form
\begin{equation} \label{eq11}
\ip{  \frac{\mathrm{d}}{\mathrm{d}t} L(t)}{\Lambda}_{\Theta} =
\ip{\mathcal{B}_{2} \Lambda \mathcal{F} X(t)}{X(t)}_{\mathbb{H}}
- \sigma \ip{L(t)}{\Lambda}_{\Theta},
\hspace{0.5em} L(0)\in \Theta,
\end{equation}
for $\Lambda \in \Theta$. For each $\Phi \in \mathbb{V}$, define the 
operator $\mathcal{M}(\Phi)  \, : \, \Theta \rightarrow \mathbb{V}^{*}$ by
\begin{equation} \label{eq12}
\ip{\mathcal{M}(\Phi)\Lambda}{\Psi}_{\mathbb{H}} =
\ip{\mathcal{B}_{2}   \Lambda   \mathcal{F}  \Phi }{\Psi}_{\mathbb{H}},
\hspace{0.25em} \Lambda\in \Theta, \Psi\in \mathbb{V},
\end{equation}
with $\mathcal{M}(\Phi) \in \mathcal{L}(\Theta,\mathbb{V}^{*})$.
For each $\Phi \in \mathbb{V}$ define its Banach space adjoint
$\mathcal{M}^{*}(\Phi) \in \mathcal{L}(\mathbb{V},\Theta)$ by
\begin{equation} \label{eq13}
\ip{\mathcal{M}^{*}(\Phi)\Psi}{\Lambda}_{\Theta}= \ip{\mathcal{M}(\Phi)\Lambda}{\Psi}_{\mathbb{H}} ,
\hspace{0.25em} \Psi\in \mathbb{V}, \Lambda\in \Theta.
\end{equation}
In view of \eqref{eq13}, the adaptation \eqref{eq11} is re-written as
\begin{equation} \label{eq14}
\ip{  \frac{\mathrm{d}}{\mathrm{d}t} L(t)}{\Lambda}_{\Theta} =
\ip{\mathcal{M}^{*}(X(t))X(t)}{\Lambda}_{\Theta}
- \sigma \ip{L(t)}{\Lambda}_{\Theta},
\hspace{0.5em} L(0)\in \Theta,
\end{equation}
with $\Lambda \in \Theta$. Using \eqref{eq12}, \eqref{eq13}, the aggregate
dynamics is given in weak form
\begin{equation} \label{eq15}
\left\{
\begin{array}{l}\ds
\ip{\frac{\mathrm{d}}{\mathrm{d}t} X(t)}{\Phi} = \ip{\mathcal{A}_{c}X(t)}{\Phi}_{\mathbb{H}}
 -\ip{\mathcal{M}(X(t))L(t)}{\Phi}_{\mathbb{H}},  
\\ \noalign{\medskip}
\ds
\ip{  \frac{\mathrm{d}}{\mathrm{d}t} L(t)}{\Lambda}_{\Theta} =
\ip{\mathcal{M}^{*}(X(t))X(t)}{\Lambda}_{\Theta}
- \sigma \ip{L(t)}{\Lambda}_{\Theta},
\\ \noalign{\medskip}
\ds
\Phi\in \mathbb{V},\hspace{0.5em} X(0)\in D(\mathcal{A}),
\hspace{0.5em} L(0),\Lambda \in \Theta,
\end{array} \right.
\end{equation}
or with $\mathcal{X}(t)=(X(t),L(t))$
\begin{equation} \label{eq16}
\begin{array}{l}
\ds \dot{\mathcal{X}}(t)   =
\left[\begin{array}{cc}
\mathcal{A}_{c} & - \mathcal{M}(X(t))[\, \bullet \, ] \\ \noalign{\medskip}
\mathcal{M}^{*}(X(t))[\, \bullet \,] &  -\sigma \mathbf{I}_{N} \end{array}\right]
\mathcal{X}(t)  
\\ \noalign{\medskip} 
\mathcal{X}(0)   \in D(\mathcal{A}) \times \Theta.
\end{array}
\end{equation}
The compact form \eqref{eq16} facilitates the well-posedness of \eqref{eq15},
as it makes use of established results on adaptive control of abstract evolution
equations (equations (2.40), (2.41) of \cite{BDRR98}). 
\begin{lem} \label{lemma1}
Consider the $N$ systems governed by \eqref{eq1} and assume that the pairs
$(A,B_{2})$ satisfy the assumption of approximately controllability
with \eqref{eq6} valid and that the state of each system
in \eqref{eq1} is available to each of its communicating
neighbors. Then the proposed synchronization controllers in
\eq{eq8} result in a closed loop system \eqref{eq9} and an adaptation law
for the edge-dependent gains \eqref{eq10} that culminate in the well-posed abstract
system \eqref{eq16} with a unique local solution
$(X,L)\in C((0,T);\mathbb{H}\times \Theta) \cap C^{1}((0,T);\mathbb{H}\times \Theta)$.
\end{lem}
\begin{proof}
The expression \eqref{eq16} is essentially in the form presented in \cite{BDRR98}.
The skew-adjoint structure of the matrix operator, which reflects the terms that cancel
out due to the adaptation, essentially facilitate 
the establishment of well-posedness. The $\Lambda$-linearity of the term
$\ip{\mathcal{B}_{2} \Lambda \mathcal{F}  \Phi }{\Psi}_{\mathbb{H}}$
along with the fact that $B_{2} \in {\mathcal L}(\mathcal{U},\mathcal{H})$,
$F\in {\mathcal L}(\mathcal{H},\mathcal{U})$ (thereby giving
$\mathcal{B}_{2} \in {\mathcal L}(\mathbb{U},\mathbb{H})$
and $\mathcal{F}\in {\mathcal L}(\mathbb{H},\mathbb{U})$) yield
$\mathcal{M}(\Phi) \in \mathcal{L}(\Theta,\mathbb{V}^{*})$. Since the assumption
on controllability gives $A_{c}$ an exponentially stable semigroup on
$\mathcal{H}$, then one has that $\mathcal{A}_{c}$ generates an
exponentially stable semigroup on $\mathbb{H}$. This allows one
to use the results in \cite{BDRR98} to establish well-posedness.
In particular, one defines $\mathfrak{X} =\mathbb{H} \times \Theta$
endowed with the inner product
$\ip{(\Phi_{1},\Lambda_{1})}{(\Phi_{2},\Lambda_{2})}_{\mathfrak{X}}=
\ip{\Phi_{1}}{\Phi_{2}}_{\mathbb{H}}+\ip{\Lambda_{1}}{\Lambda_{2}}_{\Theta}$.
Additionally, let the space $\mathfrak{Y}=\mathbb{V}\times \Theta$
endowed with the norm $\|(\Phi,\Lambda)\|_{\mathfrak{Y}}^{2}=
\|\Phi\|_{\mathbb{V}}^{2}+\|\Lambda\|_{\Theta}^{2}$. Then we have
that $\mathfrak{Y}$ is a reflexive Banach space with 
$\mathfrak{Y} \hookrightarrow \mathfrak{X} \hookrightarrow \mathfrak{Y}^{*}$.
For $\lambda >0$, the linear operator $\mathcal{A}_{\lambda} : \mathfrak{Y} \rightarrow \mathfrak{Y}^{*}$
in equation (2.40) of \cite{BDRR98} is now defined by
$\ip{\mathcal{A}_{\lambda}(X,L)}{(\Phi,\Lambda)}_{\mathfrak{Y},\mathfrak{Y}^{*}} 
=-\ip{\mathcal{A}_{c}X}{\Phi}+\ip{\lambda L}{\Lambda}_{\Theta}$ and the operator
$G_{\lambda} \, : \, \mathbb{R}^{+} \times \mathfrak{Y} \rightarrow \mathfrak{Y}^{*}$
by
$\ip{G_{\lambda}(t,X,L)}{(\Phi,\Lambda)}_{\mathfrak{Y},\mathfrak{Y}^{*}}=
-\ip{\mathcal{M}(X)L}{\Phi}_{\mathbb{H}}+\ip{\mathcal{M}^{*}(X)\Phi}{\Lambda}_{\Theta}$,
where $t>0$, $(X,\Phi)\in \mathbb{H}$, $(L,\Lambda)\in \Theta$. These then fit
the conditions in Theorem~2.4 in \cite{BDRR98}. In fact, one can extend the local
solutions for all $T>0$ and to obtain
$X \in L_{2}(0,\infty;\mathbb{V})\cap L_{\infty}(0,\infty;\mathbb{H})$,
$L\in H^{1}(0,\infty;\Theta)$ with the control signals $u_{i}\in L_{2}(0,\infty)$,
$i=1,\ldots,N$.
\end{proof}

\begin{rem}
Please note that \eqref{eq16} is used to established well-posedness, but \eqref{eq8},
\eqref{eq9} and \eqref{eq10} are used for implementation. While \eqref{eq14} avoids over
parametrization, it renders the implementation of the synchronization controllers
complex. To demonstrate this, consider scalar systems whose connectivity is described
by the undirected graph in Figure~\ref{fig0}. The aggregate closed loop systems
will need ten unknown edge-dependent gains
$\alpha_{12}$, $\alpha_{14}$, $\alpha_{15}$, $\alpha_{21}$,
$\alpha_{34}$,
$\alpha_{41}$, $\alpha_{43}$, $\alpha_{45}$,
$\alpha_{51}$, $\alpha_{54}$. 
When the Laplacian is used, the fifteen unknown entries of the Laplacian
matrix are 
$L_{11}$, $L_{12}$, $L_{14}$, $L_{15}$, 
$L_{21}$, $L_{22}$, 
$L_{33}$, $L_{34}$, 
$L_{41}$, $L_{43}$, $L_{44}$, $L_{45}$,     
$L_{51}$, $L_{54}$, $L_{55}$. 
Of course when one enforces
$L_{ii}=-\sum_{j \in N_{i}} L_{ij}$, then the number of unknown reduces to eleven.
Nonetheless, \eqref{eq11} is used for analysis and \eqref{eq10} is used
for implementation.
\end{rem}

The convergence, for both state and adaptive gains, is established in the next lemma.
\begin{lem} \label{lemma2}
For $(X,L)$ the solution to the initial value problem \eqref{eq16}, the function
$W \, : \, [0,\infty)\rightarrow \mathbb{R}^{+}$ given by
\begin{equation}
W(t) = |X(t)|_{\mathbb{H}}^{2} + \|L(t)\|^{2}_{\Theta}
\end{equation}
is nonincreasing, $X \in L_{2}(0,T;\mathbb{V}) \cap L_{\infty}(0,T;\mathbb{H})$,
$L\in L_{2}(0,T;\Theta) \cap L_{\infty}(0,T;\Theta)$, with
\begin{equation}
W(t) +  \kappa \int_{0}^{t} |X(\tau)|_{\mathbb{H}}^{2} \, \mathrm{d} \tau 
+  2 \int_{0}^{t} \|L(\tau)\|_{\Theta}^{2} \, \mathrm{d} \tau  \leq W(0),
\end{equation}
and consequently $\lim_{t\rightarrow \infty}W(t) =0$.
\end{lem}
\begin{proof}
Consider
$$
\begin{array}{lcl}
\ds  \frac{\mathrm{d}}{\mathrm{d}t}W(t) & =  &
\ds  \frac{\mathrm{d}}{\mathrm{d}t}|X(t)|_{\mathbb{H}}^{2}
+  \frac{\mathrm{d}}{\mathrm{d}t}\|L(t)\|^{2}_{\Theta} \\ \noalign{\medskip}
& = &  \ds
\ip{\frac{\mathrm{d}}{\mathrm{d}t}X(t)}{X(t)}_{\mathbb{H}}
+\ip{X(t)}{\frac{\mathrm{d}}{\mathrm{d}t}X(t)}_{\mathbb{H}} \\ \noalign{\medskip}
&  &  \ds
+ \ip{\frac{\mathrm{d}}{\mathrm{d}t}L(t)}{L(t)}_{\Theta} + \ip{L(t)}{\frac{\mathrm{d}}{\mathrm{d}t}L(t)}_{\Theta}
\\ \noalign{\medskip}
&  = & \ds \ip{\mathcal{A}_{c}X(t)}{X(t)}_{\mathbb{H}} + \ip{X(t)}{\mathcal{A}_{c}X(t)}_{\mathbb{H}}
\\ \noalign{\medskip}
&  &  \ds
 - 2 \|L(t)\|^{2}_{\Theta}  \leq   \ds
-\kappa |X(t)|^{2}_{\mathbb{H}} - 2 \|L(t)\|^{2}_{\Theta}.
\end{array}
$$
Integrating both sides from $0$ to $\infty$ we arrive at
$$
W(t) + \kappa \int_{0}^{\infty} |X(\tau)|^{2}_{\mathbb{H}} \, \mathrm{d} \tau
+ 2 \int_{0}^{t} \|L(\tau)\|^{2}_{\Theta} \, \mathrm{d} \tau \leq W(0).
$$
Application of Gronwall's lemma establishes the convergence of $W(t)$ to zero.
\end{proof}

Due to the cancellation terms in the adaptation of the interconnection strengths,
nothing specific was imposed on the synchronization gain operator other than 
$F\in {\mathcal L}(\mathcal{H},\mathcal{U})$. A simple way to choose this gain is
by setting it equal to the feedback gain $K$ and therefore one arrives at the
aggregate state equations
$$
\begin{array}{l}
\ip{\frac{\mathrm{d}}{\mathrm{d}t}X(t)}{\Phi}_{\mathbb{H}} =
\ip{\mathcal{A}X(t)}{\Phi}_{\mathbb{H}} -
\ip{\mathcal{B}_{2}\left(\mathbf{I}_{N} +L(t)\right) \mathcal{K} X(t)}{\Phi}_{\mathbb{H}}
\\ \noalign{\medskip}
\ip{  \frac{\mathrm{d}}{\mathrm{d}t} L(t)}{\Lambda}_{\Theta} =
\ip{\mathcal{B}_{2} \Lambda \mathcal{K} X(t)}{X(t)}_{\mathbb{H}}
-\ip{L(t)}{\Lambda}_{\Theta},
\end{array}
$$
for $\Lambda \in \Theta$, where $\mathcal{K}=\mathbf{I}_{N} \otimes K$.

\subsection{Optimization of constant edge-dependent synchronization gains}

Similar to the adaptive case \eqref{eq9}, the closed-loop systems with constant
edge-dependent synchronization gains are given, via \eqref{eq7}, by
\begin{equation} \label{eq19}
\dot{x}_{i}(t)=(A-B_{2}K)x_{i}(t) - B_{2}F \sum_{j\in N_{i}}\alpha_{ij} x_{ij}(t) ,
\hspace{0.5em} x_{i}(0)\in D(A),
\end{equation}
for $i=1,\ldots,N$, or in terms of the aggregate states
$$
\frac{\mathrm{d}}{\mathrm{d}t} X(t) = \mathcal{A}_{c}X(t)
-  \mathcal{B}_{2} L \mathcal{F} X(t) ,
\hspace{1em} X(0)\in D(\mathcal{A}).
$$
For simplicity, one chooses the synchronization operator gain $F$ to be identical to the
feedback operator gain $K$ and thus the above closed-loop system is written as
\begin{equation} \label{eq20}
\frac{\mathrm{d}}{\mathrm{d}t} X(t) = \Big( \mathcal{A}
-  \mathcal{B}_{2} \left( \mathbf{I}_{N}+ L \right)  \mathcal{K}\Big) X(t) ,
\hspace{1em} X(0)\in D(\mathcal{A}).
\end{equation}
The well-posedness of \eqref{eq20} can easily be established. Since the operator
$A$ generates a $C_{0}$ semigroup on $\mathcal{H}$, one can easily
argue that $\mathcal{A}$ generates a $C_{0}$ semigroup on $\mathbb{H}$.
Since $L\in \Theta$, then $\mathbf{I}_{N}+L$ is a positive definite matrix.
Consequently the operator $\mathcal{A}
-  \mathcal{B}_{2} \left( \mathbf{I}_{N}+ L \right)  \mathcal{K}$ generates
an exponentially stable $C_{0}$ semigroup on $\mathbb{H}$. Since it was assumed
that $X(0)\in D(\mathcal{A})$, the system admits a unique solution. Furthermore,
one has that $|X(t)|_{\mathbb{H}}$ asymptotically converges to zero.

A possible way to obtain the optimal values of the entries of the Laplacian matrix $L$
is to minimize an associated energy norm. The design criteria are similar to those
taken for the optimal damping distribution for elastic systems governed by second
order PDEs, \cite{FahrooACC06,FahrooCDC06}. In this case one seeks to find $L\in \Theta$
such that the associated energy of the aggregate system \eqref{eq20}, $E(t)$ satisfy
$E(t) \leq Me^{-\omega t}E(0)$, $t >0$.
 
The above is related to the stability of the closed-loop aggregate system and for
that, one needs to require that the spectrum determined growth condition is satisfied.
This condition essentially states that a system has this property of the supremum of
the real part of eigenvalues of the associated generator $\mathcal{A}
-  \mathcal{B}_{2} \left( \mathbf{I}_{N}+ L \right)  \mathcal{K}$ equals the infimum of
$\omega$ satisfying the above energy inequality.
This optimization takes the form of making the system ``more'' stable.

Related to this, an alternative criterion that is easier to implement numerically,
aims at minimizing the total
energy of the aggregate system over a long time period
$$
J=\int_{0}^{\infty}E(\tau)\, \mathrm{d} \tau,
$$
over the set of admissible (Laplacian) matrices $L \in \Theta$. This criterion
is realized through the solution to a $L$-parameterized operator Lyapunov equation
with $L$ constrained in $\Theta$, i.e. any optimal value of $L$ must
satisfy the conditions for graph Laplacian described by $\Theta$. The optimal value $L$
is then given by
$L = \arg \min _{L_{\alpha}\in \Theta} \mbox{tr } \Pi_{\alpha}$
where $\Pi_{\alpha}$ is the solution to the $L_{\alpha}$-parameterized
operator Lyapunov equation
$$
\left(  \mathcal{A}_{c} - \mathcal{B}_{2} L_{\alpha} \mathcal{K} \right)^{*} \Pi_{\alpha}
+ \Pi_{\alpha} \left(  \mathcal{A}_{c} - \mathcal{B}_{2} L_{\alpha} \mathcal{K} \right) +I=0 
\hspace{0.25em} \mbox{in $D(\mathcal{A})$},
\hspace{0.25em} L_{\alpha}\in \Theta.
$$
However, since one would like to enhance synchronization, then the cost is changed to
\begin{equation} \label{eq21}
\begin{array}{lcl}
J_{I} & = & \ds \int_{0}^{t} |X(\tau)|_{\mathbb{H}}^{2}
+ |Z(\tau)|_{\mathbb{H}}^{2} \, d \tau
\\ \noalign{\smallskip}
& = & \ds \int_{0}^{t} \ip{X(\tau)}{(I+C_{1}^{*}C_{1})X(\tau)}_{\mathbb{H}}\, d \tau.
\end{array}
\end{equation}
In view of this, the proposed optimization design is 
\begin{equation} \label{eq22}
\left\{ 
\begin{array}{l}
\mbox{\textbf{Design I: minimize \eqref{eq21} subject to \eqref{eq20} }} \\ \noalign{\smallskip}
\mbox{\textbf{Solution:} } \ds L^{opt} = \arg \min _{L_{\alpha}\in \Theta} \mbox{tr } \Pi_{\alpha}
\\ \noalign{\medskip}
\ds \left(  \mathcal{A}_{c} - \mathcal{B}_{2} L_{\alpha} \mathcal{K} \right)^{*} \Pi_{\alpha}
+ \Pi_{\alpha} \left(  \mathcal{A}_{c} - \mathcal{B}_{2} L_{\alpha} \mathcal{K} \right) 
\\ \noalign{\smallskip}
\hspace*{1em} 
+(I+C_{1}^{*}C_{1})=0,
\hspace{0.5em} L_{\alpha}\in \Theta.
\end{array}  \right.
\end{equation}

The optimization \eqref{eq22} above does not account for the cost of the control law.
If the structure of the control law \eqref{eq7} is assumed with $F=K$ and $K$ chosen such
that $A-B_{2}K$ generates an exponentially stable $C_{0}$ semigroup on $\mathcal{H}$,
then one may consider the effects of the control cost when searching for the optimal
value of the constant graph Laplacian matrix $L$. In this case, the cost
functional in \eqref{eq21} is now modified to 
\begin{equation} \label{eq23}
J_{II} = \int_{0}^{t} \ip{X(\tau)}{(I+C_{1}^{*}C_{1})X(\tau)}_{\mathbb{H}}
+ |U(\tau)|^{2}_{\mathbb{U}}\, d \tau.
\end{equation}
Please note that the term $\ip{U(\tau)}{U(\tau)}_{\mathbb{U}}$ is given explicitly by
$\ip{U(t)}{U(t)}_{\mathbb{U}} = \sum_{i=1}^{N}\ip{u_{i}(t)}{u_{i}(t)}_{\mathcal{U}}$
where $u_{i}$ are given by \eqref{eq7}. Unlike \eqref{eq22}, \eqref{eq23},
the optimization for $L$ in this case must be performed numerically and the optimal value
is
\begin{equation} \label{eq24}
\left\{ 
\begin{array}{l}
\mbox{\textbf{Design II: minimize \eqref{eq23} subject to \eqref{eq20} }}
\\ \noalign{\smallskip} 
\mbox{\textbf{Solution:} } \ds L^{opt} = \arg \min _{L_{\alpha}\in \Theta} J_{II}  . 
\end{array} \right.
\end{equation}

To consider an optimal control for the aggregate system, without assuming a specific
structure of the controller gains $K$ and $F$, but with a prescribed constant graph
Laplacian matrix, one may be able to pose the synchronization problem as 
an optimal (linear quadratic) control problem.
One rewrites \eqref{eq20} without
the assumption that the synchronization operator gain $F$ is equal to
the regulation operator gain $K$. Thus \eqref{eq19} when written is aggregate form
produces
\begin{equation} \label{eq25}
\begin{array}{lcl}
\frac{\mathrm{d}}{\mathrm{d}t} X(t)  & = & \ds    \mathcal{A} X(t)
-  \mathcal{B}_{2} \mathcal{K}  X(t)  -  \mathcal{B}_{2} L \mathcal{F}  X(t)  \\ \noalign{\medskip}
& = & \ds    \mathcal{A} X(t)
+ \mathcal{B}_{2} U_{1}(t)  +  \mathcal{B}_{2} L U_{2}(t) \\ \noalign{\medskip}
& = & \ds    \mathcal{A} X(t)
+\widetilde{\mathcal{B}}_{2} U(t)
\end{array}
\end{equation}
for $X(0)\in D(\mathcal{A})$,
where the augmented input operator and augmented control signal are given by
$$
\widetilde{\mathcal{B}}_{2} =\mathcal{B}_{2} \left[\begin{array}{cc} \mathbf{I}_{N} & L \end{array}\right],
\hspace{2em}
\widetilde{U}(t) = \left[\begin{array}{c} U_{1}(t) \\ \noalign{\smallskip} U_{2}(t) \end{array}\right].
$$
One can then formulate an optimal control policy for 
the aggregate system in $\mathbb{H}$
\begin{equation} \label{eq26}
\frac{\mathrm{d}}{\mathrm{d}t} X(t)  = \mathcal{A} X(t)
+ \widetilde{\mathcal{B}}_{2} \widetilde{U}(t),
\hspace{1em} X(0)\in D(\mathcal{A}),
\end{equation}
as follows: Find $\widetilde{U}$ such that the cost functional 
\begin{equation}  \label{eq27}
J_{III} = \int_{0}^{t} \ip{X(\tau)}{(I+C_{1}^{*}C_{1})X(\tau)}_{\mathbb{H}}
+ |U_{1}(\tau)|^{2}_{\mathbb{U}}
+ |U_{2}(\tau)|^{2}_{\mathbb{U}}\, d \tau.
\end{equation}
is minimized. The solution to this LQR problem is given by
\begin{equation} \label{eq28}
\left\{ 
\begin{array}{l}
\mbox{\textbf{Design III: minimize \eqref{eq27} subject to \eqref{eq26}}}
\\ \noalign{\smallskip}
\mbox{\textbf{Solution:} }
\widetilde{U}(t) = - \widetilde{\mathcal{B}}_{2}^{*} \mathcal{P} X(t) \\ \noalign{\medskip}
\mathcal{A}^{*} \mathcal{P} + \mathcal{P}\mathcal{A} 
- \mathcal{P}\widetilde{\mathcal{B}}_{2}^{*}\widetilde{\mathcal{B}}_{2}\mathcal{P}
 +(I+C_{1}^{*}C_{1})=0.
\end{array}  \right.
\end{equation}
In the event that any form of optimization
for the edge-dependent gains cannot be performed, then a static optimization
can be used; in this case, the edge-dependent gains can be chosen in
proportion to the pairwise state mismatches
$\alpha_{ij} =  |x_{ij}(0)|_{\mathcal{H}}$.

\begin{rem}
In the case of full connectivity, thereby simplifying the Laplacian to
$L=N\mathbf{I}_{N}- \mathbf{1} \cdot \mathbf{1}^{T}$
and when the edge-dependent gains are all identical $\alpha_{ij}=\alpha$, then
one may be able to obtain an expression for the dynamics
of the pairwise differences $x_{ij}=x_{i}-x_{j}$
$$
\dot{x}_{ij}(t) = A_{c}x_{ij}(t) - \alpha N BF x_{ij}(t),
\hspace{1em} x_{ij}(0)\ne 0.
$$
As was pointed out in \cite{DemetriouSCL2013}, when the system operator
is Riesz-spectral and certain conditions on the input operator
are satisfied, one can obtain explicit bounds on the exponential convergence
of $x_{ij}$ (in an appropriate norm) to zero. Additionally for this
case one has that the convergence of $x_{ij}$ is faster than that of
$x_{i}$  and is a function of $\alpha$.
\end{rem}

\section{Numerical studies}\label{sec4}

The following 1D diffusion PDE was considered  
$$
\frac{\partial x_{i}}{\partial t}(t,\xi)=a_{1}\frac{\partial^{2}x_{i}}{\partial t^{2}}(t,\xi)
+b(\xi)u_{i}(t),
\hspace{0.5em} x(t,0)=x(t,1)=0,
$$
The control distribution function $b(\xi)$ was taken to be the approximation
of the pulse function centered at the middle of the spatial domain $[0,1]$
and $a_{1}=0.05$. Using a finite element approximation scheme with
$40$ linear splines, the system was simulated using the ode suite in
Matlab$^\circledR$. A total of $N=5$ networked systems were considered and whose
communication topology was described by the graph in Figure~\ref{fig0}. The feedback gain
was taken to be $K\phi =5\times 10^{-4}\int_{0}^{1} b(\xi) \phi(\xi) \, \mathrm{d} \xi$
and the synchronization gain $F=20K$. The initial conditions for the $5$ networked systems
were taken to be
$x_{1}(0,\xi)=39.4 \sin(1.3\pi \xi) e^{-7\xi^2}$,
$x_{2}(0,\xi)= 12.6 \sin(2.1 \pi \xi) \cos(1.5 \pi \xi)$,
$x_{3}(0,\xi)= 7.6 \sin(3.6\pi\xi) e^{-7\xi^2}$,
$x_{4}(0,\xi)= 2.5\sin(5\pi\xi) e^{-\xi^2}$,
$x_{5}(0,\xi)= -26.2\sin(5.\pi\xi) e^{-7 (\xi-0.5)^2}$.
\begin{figure}[ht]
\begin{center}
\includegraphics[width=5.50cm]{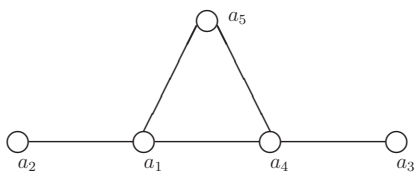}
\caption{Undirected graph on five vertices (PDE systems).}
\label{fig0}
\end{center}
\end{figure}

\subsection{Constant edge-dependent synchronization gains}

Using the control laws
$u_{i}(t) = -k(\xi)x_{i}(t,\xi) - f(\xi) \alpha \sum_{j\in N_{i}}x_{ij}(t,\xi)$
the performance functional was taken to be
\begin{equation} \label{eq30}
J_{II}=\int_{0}^{2} \,\, \sum_{i=1}^{5} \,\, \left(  |x_{i}(\tau,\cdot)|_{L_{2}}^{2}+
|z_{i}(\tau,\cdot)|_{L_{2}}^{2} + u_{i}^{2}(\tau) \right) \, \mathrm{d} \tau,
\end{equation}
where $z_{i}(t,\xi)=x_{i}(t,\xi)-\sum_{j=1}^{5}x_{j}(t,\xi)/5$.
The performance functional was evaluated for a range of values of
the uniform synchronization gain $\alpha$ in the interval $[0,2]$. This
cost is depicted in Figure~\ref{fig1a} and its optimal value 
is attained when $\alpha=0.3$. 

To examine the effects of the synchronization gain,
the norm of the aggregate deviation from the mean was evaluated for
$\alpha=0,0.3$ and $\alpha=2$. The evolution of $|Z(t)|_{\mathbb{H}}$ 
is depicted in Figure~\ref{fig1b}. As expected, the higher the value
of $\alpha$, the faster the convergence. However, only the value $\alpha=0.3$
results in acceptable levels of the deviation from the mean and low values of
the control cost.

\begin{figure*}
\centerline{
\subfigure[Effects of $\alpha$ on the performance functional \eqref{eq27}.]
{\includegraphics[width=8.25cm]{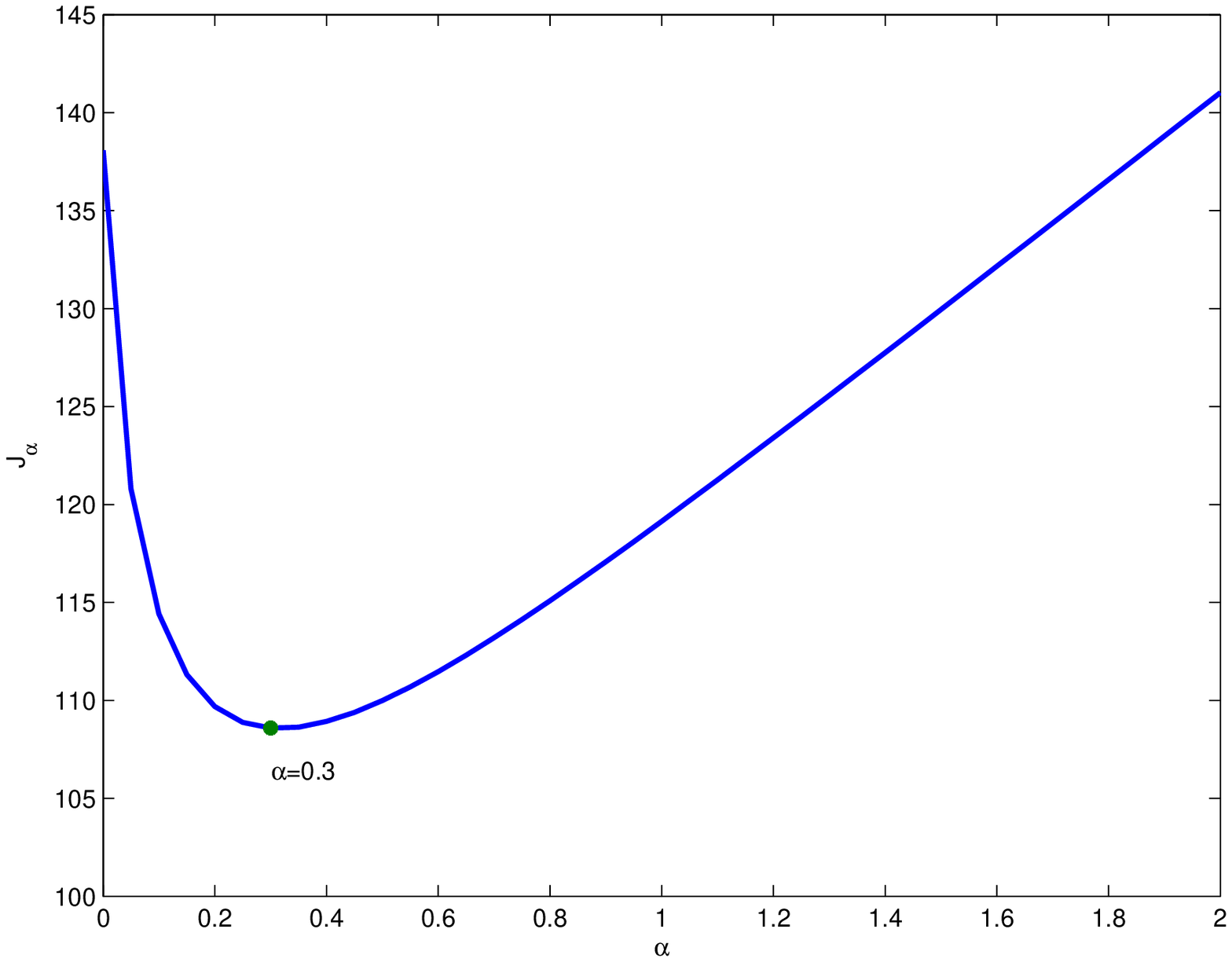}
\label{fig1a} }   \hfil
\subfigure[Norm of the deviation from the mean.]
{\includegraphics[width=8.25cm]{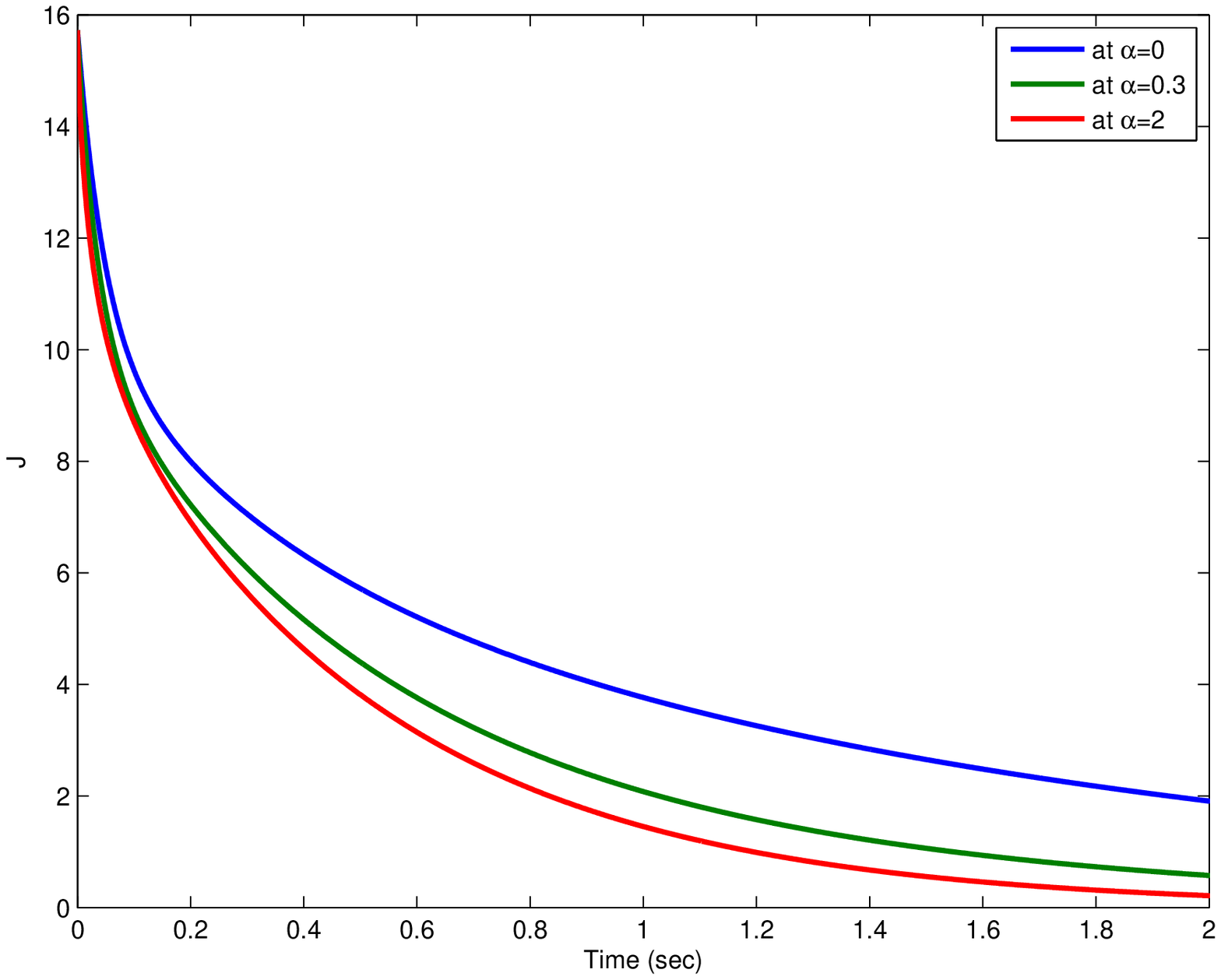}
\label{fig1b} } }
\caption{Constant edge-dependent gains.} \label{fig1}
\end{figure*}

\subsection{Adaptive edge-dependent synchronization gains}

The adaptive controller \eqref{eq8} was applied to the PDE with
$a_{1}=0.1$, $K\phi =5\times 10^{-4}\int_{0}^{1} b(\xi) \phi(\xi) \, \mathrm{d} \xi$,
$F=2K$.
The same initial conditions as in the constant gain case were used.
The adaptations in \eqref{eq10} were implemented with
an adaptive gain of $100$ and $\sigma=10^{-5}$, i.e.
$$
\dot{\alpha}_{ij} =\gamma\left[ \left( \int_{0}^{1}b(\xi)x_{i}(t,\xi)\, \mathrm{d} \xi \right)
 \left( \int_{0}^{1}f(\xi) x_{ij}(t,\xi) \, \mathrm{d} \xi \right)
-  \sigma \alpha_{ij} \right],
$$
Figure~\ref{fig2a} compares 
the adaptive to the constant edge-dependent gains case.
The initial guesses of the adaptive edge-dependent gains
were all taken to be $\alpha_{ij}(0)=1$. The same values 
of $\alpha_{ij}=1$ were used
for the constant case. The norm of the aggregate
deviation from the mean exhibits an improved convergence to zero when
adaptation of the edge-dependent gains is implemented.

The spatial distribution of the mean state ($x_{m}(t,\xi)=\sum_{i=1}^{5}x_{i}(t,\xi)/5$)
is depicted at the final time $t=2$ for both the adaptive and constant gains case in
Figure~\ref{fig2b}. It is observed that when adaptation is implemented,
the mean state converges (pointwise) to zero faster than the constant case.  
\begin{figure*}
\centerline{
\subfigure[Adaptive vs constant edge-dependent gain on $|Z(t)|_{\mathbb{H}}$.]
{\includegraphics[width=8.25cm]{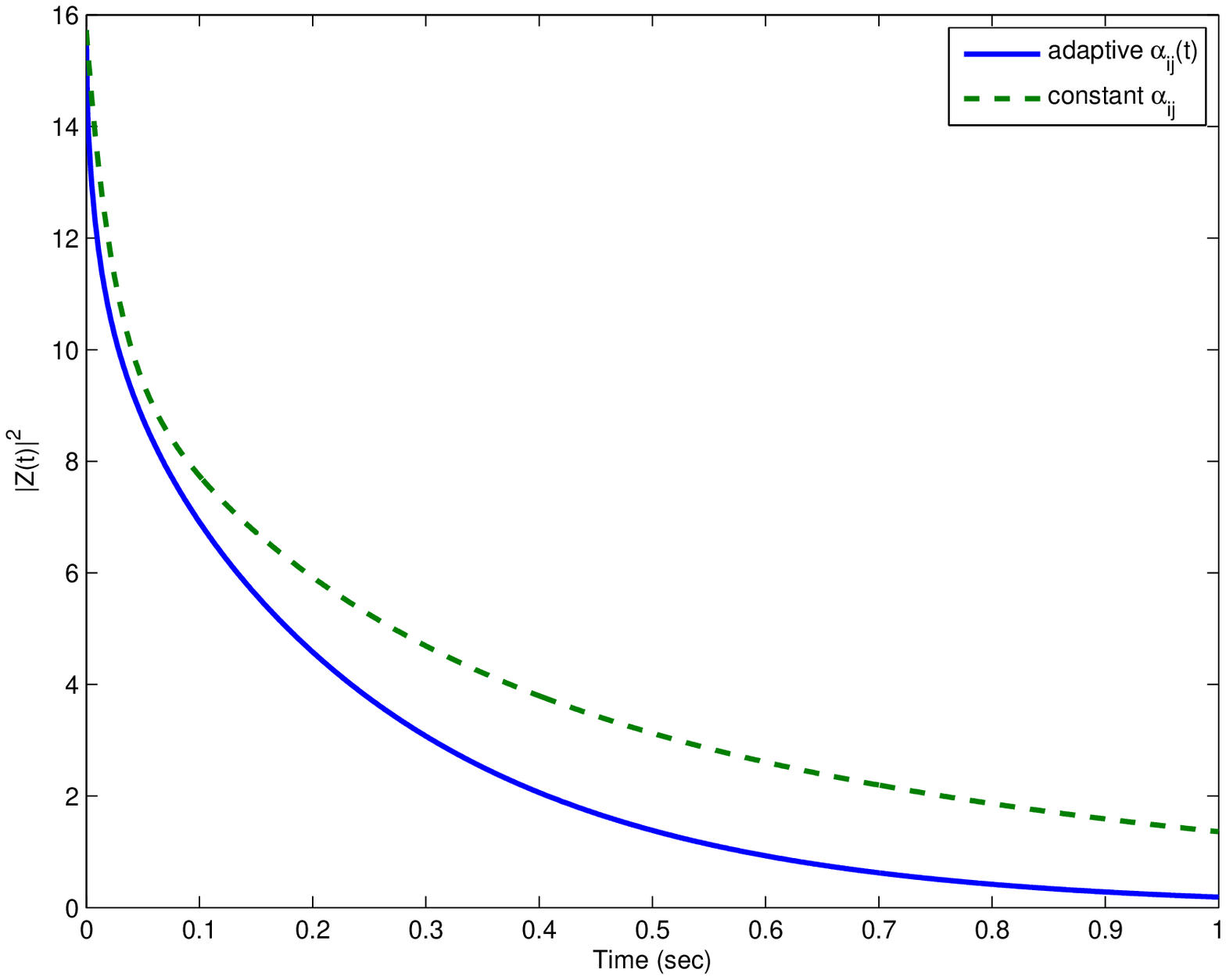}
\label{fig2a} }   \hfil
\subfigure[Spatial distribution of mean state at final time.]
{\includegraphics[width=8.25cm]{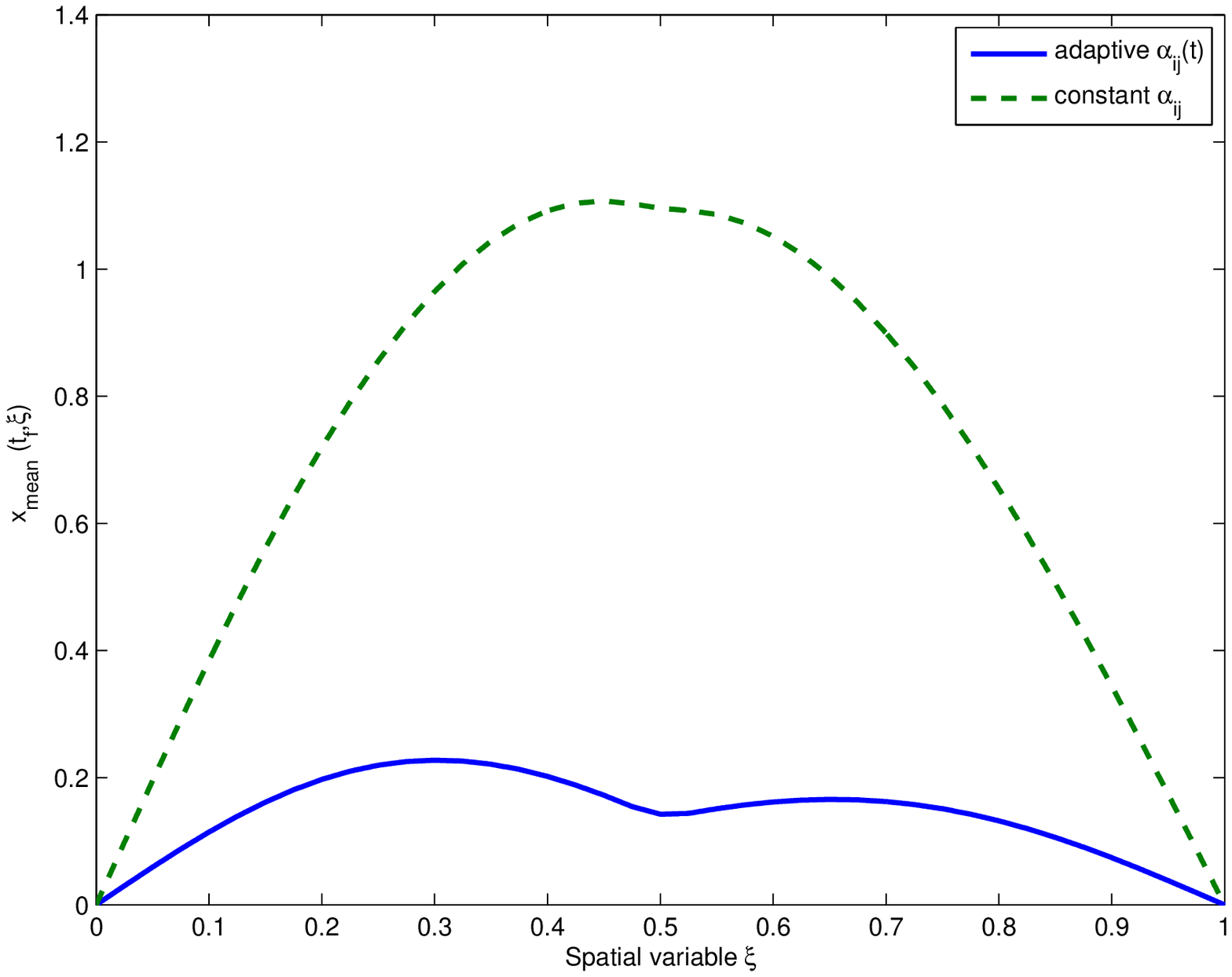}
\label{fig2b} } }
\caption{Adaptive edge-dependent gains.} \label{fig2}
\end{figure*}

\section{Concluding remarks}\label{sec5}

In this note, a scheme for the adaptation of the synchronization gains used in
the synchronization control of a class of networked DPS was proposed. The
same framework allowed for the optimization
of constant edge-dependent gains which was formulated as an optimal
(linear quadratic) control problem of the associated aggregate system of
the networked DPS. The proposed scheme
required knowledge of the full state of the networked DPS.
Such a case represents a baseline for the synchronization of networked
DPS. The subsequent extension to output feedback, whereby each networked DPS
can only transmit and receive partial state information provided
by sensor measurement, will utilize the same abstract framework presented here. 



\end{document}